\documentclass[12pt]{amsart}
\usepackage[utf8]{inputenc} % set input encoding (not needed with XeLaTeX)
\usepackage{amsthm}
\usepackage{amssymb}
\usepackage{fancyref}
\usepackage{array} % for better arrays (eg matrices) in maths
\usepackage{ dsfont }
\usepackage{color, soul}
\usepackage{ bbold }

\usepackage[english]{babel}

\newtheorem{theorem}{Theorem}[section]

\newtheorem{lemma}{Lemma}[section]
\newtheorem{prop}{Proposition}[section]
\newtheorem{rem}{Remark}[section]

\usepackage[left=3cm,right=3cm,top=3cm,bottom=3cm]{geometry}
\usepackage{fancyhdr} % This should be set AFTER setting up the page geometry
\newcommand{\R}{\mathds{R}}
\newcommand{\E}{\mathds{E}}
\newcommand{\PP}{\mathds{P}}
\newcommand{\sS}{\mathds{S}}
\newcommand{\dint}{\mathrm{d}}

\usepackage{mathtools}
\allowdisplaybreaks

\begin{document} 

\title{Thin-shell concentration for zero cells of stationary Poisson mosaics}

%    Remove any unused author tags.

%    author one information
\author{Eliza O'Reilly}
\address{University of Texas at Austin, Department of Mathematics, RLM 8.100,
2515 Speedway Stop C1200
Austin, Texas 78712-1202}
\curraddr{}
\email{eoreilly@math.utexas.edu}
\thanks{The author was supported by a grant of the Simons Foundation (\#197982 to UT Austin) and by the National Science Foundation Graduate Research Fellowship under Grant No. DGE-1110007.}



\date{}

\dedicatory{}

\begin{abstract}

We study the concentration of the norm of a random vector $Y$ uniformly sampled in the centered zero cell of two types of stationary and isotropic random mosaics in $\R^n$ for large dimensions $n$. For a stationary and isotropic Poisson-Voronoi mosaic, $Y$ has a radial and log-concave distribution, implying that ${|Y|}/{\mathds{E}(|Y|^2)^{\frac{1}{2}}}$ approaches one for large $n$. Assuming the cell intensity of the random mosaic scales like $e^{n \rho_n}$, where $\lim_{n \to \infty} \rho_n = \rho$, $|Y|$ is on the order of $\sqrt{n}$ for large $n$. For the Poisson-Voronoi mosaic, we show that $|Y|/\sqrt{n}$ concentrates to $e^{-\rho}(2\pi e)^{-\frac{1}{2}}$ as $n$ increases, and for a stationary and isotropic Poisson hyperplane mosaic, we show there is a range $(R_{\ell}, R_u)$ such that ${|Y|}/{\sqrt{n}}$ will be within this range with high probability for large $n$. The rates of convergence are also computed in both cases.%if we condition on the zero cell having inradius greater than or equal to some positive $r_c$. 
%These results indicate that in high dimensions, the bulk of the volume of these zero cells is concentrated in a thin annulus at distance $\mathds{E}[|Y|^2]^{\frac{1}{2}}$ from the center of the cell. 

\end{abstract}

\maketitle

\section{Introduction}

%It has been observed in a number are mathematical fields, that it can be easier to prove that a random object satisfies some desirable property with high, or non-zero, probability, than it is to construct an object deterministically that has that characteristic. 
%There have been many recent developments in the theory of random polytopes, pushed by applications in areas such as optimization \cite{simplex}, random matrices \cite{rudelson}, and convex geometric analysis \cite{klartag}.  Often the goal is to show that a random polytope will have some geometric property with high, or non-zero probability, that is difficult to show for a deterministically constructed polytope. 
Random mosaics, also called random tessellations, have long been studied in stochastic geometry and give rise to interesting classes of random convex polytopes. 
Some well-known classes of random mosaics are built from Poisson point processes, either in $\R^n$ or in the space of hyperplanes in $\R^n$. Statistics of the cells of these random mosaics have been well-studied, particularly in dimensions $n=2$ and $n=3$. See \cite{newSG} and \cite[Chapter 10]{weil} for more background and further references. Particular attention has been paid to the zero cell and the typical cell, two random convex polytopes induced by the random mosaic. The zero cell is the cell of the mosaic containing the origin, and the distribution of the typical cell is obtained by averaging over all cells in a large bounded subset and then increasing this subset to the entire space. %In fixed dimension, the asymptotic behavior of the volume and shape and these cells has been studied in  series of papers \cite{HugZero}. 

Recently, there has been more interest in high dimensional random tessellations, partially due to applications in signal processing \cite{PV1} and information theory \cite{venkat}. For these applications, it is important to understand the asymptotic geometric properties of the convex polytopes induced by random tessellations, in order to decode and reconstruct high dimensional signals with small error.
Some recent work has considered high dimensional Poisson mosaics in particular, focusing on the volume and shape of the zero cell and typical cell as dimension $n$ tends to infinity \cite{Voronoi, HorrmanZero, HugPoly}.  %The volume of these cells has been studied in \cite{Voronoi} and \cite{HorrmanZero}, as well as some analysis of their shape in high dimensions. 
For example, in \cite{Voronoi}, it is proved that the volume of the intersection of the typical cell of a Poisson-Voronoi mosaic with intensity $\lambda$ and a co-centered ball of volume $u$ tends to $\lambda^{-1}(1-e^{-\lambda u})$ as the space dimension tends to infinity. % where $\lambda$ is the intensity of the underlying Poisson point process.  
%In \cite{HugPoly}, the authors show that asymptotically almost surely, the normalized zero cells of some isotropic Poisson mosaics (as well as the typical cell of a stationary Poisson-Voronoi mosaic) satisfy the hyperplane conjecture (see \cite[Chapter 3]{CVBook}), one of the major open problems in asymptotic convex geometry. 

In this paper, we aim to better understand the nature of the zero cell of stationary Poisson mosaics in high dimensions by considering a phenomenon in asymptotic convex geometry called thin-shell concentration. 
%Since the cells of random mosaics are convex polytopes, one can also study their properties in relation to results in convex geometric analysis. 
It is well-known (see e.g. \cite{bobkov}) that a radial random vector $Y$ in $\R^n$ with density $f(x) := g(|x|)$, where $g: \R \rightarrow \R$ is log-concave, satisfies the following: For some absolute constant $C > 0$, for all $t \geq 0$,
\begin{align}\label{e:thinshell}
\mathds{E}\left(\frac{|Y|}{(\mathds{E}|Y|^2)^{\frac{1}{2}}} - 1 \right)^2 \leq \frac{C}{n}.
\end{align}
This inequality implies the norm of $Y$ will be concentrated in a ``thin-shell" around its expectation for large dimension $n$, and is called a thin-shell estimate. %, a phenomenon known as thin-shell concentration. %and \eqref{e:thinshell} is called a thin-shell estimate. 
A major open problem in asymptotic convex geometry is the thin-shell conjecture, which claims that \eqref{e:thinshell} holds for all log-concave random vectors $Y$ in $\R^n$ normalized so that $\mathds{E}(Y) = 0$ and $\mathds{E}(Y_iY_j) = \delta_{ij}$, for $i, j = 1, \ldots ,n$. %An equivalent statement of the conjecture is the following deviation inequality:
%For some absolute constant $C > 0$, for any $t > 0$, 
%   \[\PP \bigg(\big|\|Y\|_2 - \sqrt{n}\big| \geq t \bigg) \leq e^{- C t} .\]
The best currently known thin-shell estimate follows from the following deviation estimate: for absolute constants $C, c >0$,
\[\mathds{P}(| |Y | - \sqrt{n} | \geq t \sqrt{n}) \leq Ce^{-c\min{\{t^3, t\}}\sqrt{n}} \, \, \text{ for all } t \geq 0,\]
and was proved by Gu\'edon and Milman in \cite{Milman}. 
%A famous conjecture in convex geometric analysis, known as the thin-shell conjecture, states that for any log-concave random vector $Y$ such that $\mathds{E}(Y) = 0$ and $\mathds{E}(Y \otimes Y) = I_n$, where $I_n$ is the square $n \times n$ identity matrix, there exists an absolute constant $c > 0$ such that, for any $t > 0$,
%\[\mathds{P}(| |Y | - \sqrt{n} | \geq t \sqrt{n}) \leq 2e^{-ct\sqrt{n}}.\]
We refer to the monograph \cite{CVBook} for more on thin-shell estimates and log-concave random vectors. %and a summary of extensions and improvements of the estimate \eqref{e:thinshell}.
%The uniform random measure on a convex set is log-concave, but it is not necessarily the case when this is only the conditional probability measure. 

In the following, we study to what extent the phenomenon of thin-shell concentration occurs for the random vector that, conditioned on the random mosaic, is uniformly distributed in the centered zero cell. By ``centered", we mean that an appropriately chosen center of the cell is located at the origin, for instance, the center of the largest ball contained within the cell. If this random vector is concentrated around its mean in high dimensions, then most of the volume of the zero cell of the random mosaic will be contained within a narrow annulus around its center. 

%We study the random vector that, conditioned on the random mosaic, is uniformly distributed in the centered zero cell. By ``centered", we mean that an appropriately chosen center of the cell is located at the origin, for instance, the center of the largest ball contained within the cell. 

One motivation for the study of the norm of this random vector is data compression. Random mosaics can be used to compress data in $\R^n$ such that all data contained in the same cell of the tessellation will have the same encoding. This is the case, for instance, in one-bit compressed sensing using hyperplane tessellations, see \cite{PV1} and \cite{PV2}. Reconstructing the original data with small error requires that all data within the same cell of the tessellation are close together. The volume of the cell is not a useful metric in this case, since a very thin cell could have small volume and also contain signals that lie very far apart. The norm of the random vector studied in this paper is a more useful metric to ensure the mass of the cell does not lie far away from the center.

Our study of the random vector chosen uniformly from the centered zero cell begins with the observation is that its distribution is shown to depend on the typical cell, as shown in Lemma \ref{X_density}. This is due to the fact that the distribution of the zero cell has a Radon-Nikodym derivative with respect to the distribution of the typical cell. We then restrict to studying two types of stationary random mosaics, a stationary Poisson-Voronoi mosaic and a stationary and isotropic Poisson hyperplane mosaic, since in both cases there exists an explicit representation for the distribution of the typical cell that allows for computations. Both of these random mosaics are isotropic, that is, their distribution is invariant under rotations about the origin. This implies that the random vector of interest will be radially symmetric. In the Poisson-Voronoi case, we show that this random vector is also log-concave, and thus satisfies the thin-shell estimate \eqref{e:thinshell}. We also prove strong deviation estimates by direct computation. 

The main complementary results can be stated as follows. For each $n$, let $X_n$ be a stationary random mosaic in $\R^n$ where the intensity of cell centroids is $e^{n\rho_n}$ and assume $\lim_{n \to \infty} \rho_n = \rho \in \R$. 
%In the following, the intensity of the point process of centers of the cells in the random mosaic is called the cell intensity. 
%These results indicate that the cells of a Poisson-Voronoi tessellation maybe more regularly shaped and have mass distributed similarly to a sphere in high dimensions. Our results align with this observation. 
%In particular, the authors in \cite{Voronoi} comment that the concentration of measure phenomenon makes it difficult form them to interpret this result, the uniform random vector in the zero cell showcases how the concentration of measure phenomenon in high dimensions plays out in these high dimensional random convex sets. 
Let $Y_n$ denote a random vector in $\R^n$ such that, conditionally on $X_n$, $Y_n$ is uniformly distributed in the centered zero cell of $X_n$. For the Poisson-Voronoi mosaic, we show that ${|Y_n|}/{\sqrt{n}}$ concentrates to $e^{-\rho}(2\pi e)^{-\frac{1}{2}}$ as the dimension $n$ increases. %, where the intensity of the Poisson point process that generates the random mosaic in $\R^n$ is $e^{n\lambda}$. 
Exponential rates of convergence are also computed, as shown in Theorem \ref{t:vor_thresh}.
%\begin{theorem}\label{t:main_vor}
%For each $n$, let $Y_n \sim Uniform(Z_{0, n} - c(Z_{0, n}))$, where $Z_{0,n}$ is the zero cell of a stationary Poisson Voronoi mosaic in $\R^n$ with cell intensity $\lambda_n = e^{n \lambda}$ for $\lambda \in \mathds{R}$. Then
%\begin{align*}
%\lim_{n \rightarrow \infty} \mathds{P}(|Y_n| \leq \sqrt{n}R) \rightarrow \begin{cases}
%0, & R < e^{-\lambda}\left(2\pi e \right)^{-\frac{1}{2}} \\ 1, & R > e^{-\lambda}\left(2\pi e \right)^{-\frac{1}{2}} \end{cases}.
%\end{align*}
%For $R < e^{-\lambda}\left(2\pi e \right)^{-\frac{1}{2}}$,
%\begin{align*}
%\lim_{n \rightarrow \infty} \frac{1}{n} \ln \mathds{P}(|Y_n| \leq \sqrt{n}R) = \lambda + \frac{1}{2} \ln (2 \pi e) + \ln R,
%\end{align*}
%and for $R > e^{-\lambda}\left(2\pi e \right)^{-\frac{1}{2}}$,
%\begin{align*}
%\lim_{n \rightarrow \infty} \frac{1}{n} \ln \left( - \frac{1}{n} \ln \mathds{P}(|Y_n| \geq \sqrt{n}R) \right) = \lambda + \frac{1}{2} \ln (2 \pi e) + \ln R.
%\end{align*}
%\end{theorem}
%The zero cell in a isotropic Poisson hyperplane tessllations in high dimensions was studied in \cite{HugZero} and \cite{HugPoly}. In both cases, the main metric that was studied was the volume. Since two cells with the same volume can have very different circumradii, this is not that informative about the distance of the volume of the cell from the center.
In the case of the zero cell of a Poisson hyperplane tessellation, we show there exists an interval $(R_{\ell}, R_u)$ such that ${|Y_n|}/{\sqrt{n}}$ will be contained in this interval with high probability for large dimension $n$. Rates of convergence are also computed in this case as shown in Theorem \ref{t:hyp_thresh}.

%%%%%%%%%%%%%%%%%%%%%%%%%%%%%%%%%%%%%%%%%%%%%%%

\section{Preliminaries and notation}

%%%%%%%%%%%%%%%%%%%%%%%%%%%%%%%%%%%%%%%%%%%%%%%

Let $\mathcal{F}$ denote the set of closed sets in $\R^n$ and define $\mathcal{F}' := \mathcal{F} \backslash \emptyset$. Similarly, define $\mathcal{C}$, $\mathcal{K}$, $\mathcal{C}'$, and $\mathcal{K}'$ as the set of compact sets in $\R^n$, the set of compact and convex sets of $\R^n$, and their non-empty counterparts, respectively. Random sets will be studied with hitting probabilities using the following notation. For $A, B \subset \R^n$, define
\[\mathcal{F}_A:= \{ F \in \mathcal{F} : F \cap A \neq \emptyset\}, \qquad \mathcal{F}^B := \{F \in \mathcal{F} : F \cap B  = \emptyset\},\]
and
\[ \mathcal{F}_A^B := \mathcal{F}_A \cap \mathcal{F}^B.\] % \{F \in \mathcal{F} : F \cap A \neq \emptyset, F \cap B = \emptyset\}.\]

The open ball in $\R^n$ of radius $R$ centered at the origin is denoted by $B_n(R)$ and the unit sphere by $\mathds{S}^{n-1}$. The measure $\sigma_{n-1}$ will denote the uniform probability measure on $\mathds{S}^{n-1}$, i.e., the normalized spherical Lebesgue measure.
Also, let $\kappa_n$ denote the volume of the unit ball $B_n(1)$, and $\omega_n$ the surface area of the unit sphere $\mathds{S}^{n-1}$. Note that $\omega_n = n\kappa_n$ and $\kappa_n = \frac{\pi^{n/2}}{\Gamma(n/2 +1 )}$, where $\Gamma(x) = \int_0^{\infty} e^{-t} t^{x-1}\dint t$ is the gamma function. Stirling's formula implies the following asymptotics that will be used throughout the paper: as $n \to \infty$,
%\begin{equation}\label{stirling}
%\Gamma(x) \sim \sqrt{2\pi x} \left(\frac{x}{e}\right)^x,
%\end{equation}
%and this implies that as $n \to \infty$, 
\begin{equation}\label{e:vball}
\kappa_n \sim \frac{1}{\sqrt{\pi n}}\left(\frac{2\pi e}{n}\right)^{n/2}, \qquad \kappa_n^{\frac{1}{n}} \sim \sqrt{\frac{2\pi e}{n}}, \qquad \text{and} \qquad \frac{n \kappa_n}{\kappa_{n-1}} \sim \sqrt{2\pi n} .
\end{equation}

A mosaic is defined to be a collection of convex polytopes in $\R^n$ such that the union is the entire space and no two polytopes in the collection share interior points. Let $\mathds{M}$ denote the set of all face-to-face mosaics (see \cite[Section 10.1]{weil}). Then, a random mosaic in $\R^n$ is defined to be a particle process in $\R^n$, that is, a point process in $\mathcal{C}'$, such that $X \in \mathds{M}$ almost surely. The polytopes contained in the mosaic will be referred to as the cells of the mosaic.

The intensity measure of a stationary particle process $X$ is defined as $\Theta(\cdot) := \E[X(\cdot)]$ and can be decomposed in the following way.  Let $c: \mathcal{C}^{'} \rightarrow \R^n$ be a {\em center function}, a measurable map which is compatible with translations, i.e. $c(C + x) = c(C) + x$ for all $x \in \R^n$. Define the grain space
\[\mathcal{C}_0 := \{ C \in \mathcal{C}' : c(C) = 0\},\]
and the homeomorphism (see \cite[Section 4.1]{weil} for more details)
\[\Phi: \R^n \times \mathcal{C}_0 \rightarrow \mathcal{C}' ; \qquad (x, C) \rightarrow x + C.\]

\begin{theorem} (Theorem 4.1.1 in \cite{weil}) Let $X$ be a stationary particle process in $\R^n$ with intensity measure $\Theta \neq 0$. Then there exists a number $\lambda \in (0, \infty)$ and a probability measure $\mathds{Q}$ on $\mathcal{C}_0$ such that 
\[\Theta = \lambda \Phi ( \nu \otimes \mathds{Q}).\]
\end{theorem}
The number $\lambda$ is called the intensity of the particle process and it will also be referred to as the cell intensity. $\mathds{Q}$ is called the grain distribution. The point process of centers of the cells of a stationary mosaic is a stationary point process in $\R^n$ with intensity $\lambda$. For a stationary random mosaic $X$, a random set with distribution $\mathds{Q}$ is called the typical cell of $X$.
%The distribution of a point chosen uniformly in the zero cell of a stationary random mosaic will depend on the typical cell of the tessellation. 
%This cell can be considered as a spatial average of the cells in the mosaic, as it can also be defined as follows.
%\begin{definition}
%The typical cell $Z$ of a random mosaic $X$ with intensity $\lambda$ is the random polytope with the following distribution. For all Borel sets $\mathcal{A} \in \mathcal{B}(\mathcal{K})$,
%\begin{align*}
%\mathds{Q}_0(\mathcal{A}) = \frac{1}{\hat{\gamma}|B|} \mathds{E} \sum_{P \in \hat{X}} 1_{\mathcal{A}}\{P - c(P)\}1_{B}(c(P))
%\end{align*}
%and ergodic interpretation
%\begin{align*}
%\mathds{Q}(\mathcal{A})  = \lim_{r \rightarrow \infty} \frac{\mathds{E}\left[\sum_{K \in X} 1_{\mathcal{A}}\{K - c(K)\}1_{B_n(r)}(c(K))\right]}{\mathds{E}\left[\sum_{K \in X} 1_{B_n(r)}(c(K))\right]},\quad a.s.
%\end{align*}
%where $B^n$ is the unit ball in $\R^n$.
%\end{definition}
It is known that that the expected volume of the typical cell is given by the reciprocal of the cell intensity, i.e.,
\begin{align}\label{e:EVZ}
\mathds{E}[V(Z)] = \int V(K) \mathds{Q}(K) = \frac{1}{\lambda}.% =\frac{1}{\kappa_n}\left( \frac{n \kappa_n }{ \gamma\kappa_{n-1}}\right)^n. 
\end{align}

The zero cell of a random mosaic, denoted by $Z_0$, is defined to be the cell the origin is contained in. Another interpretation of the typical cell is the distribution of the zero cell under the Palm distribution $\PP^0$ of the cell centers (see \cite[Chapter 3]{weil}). That is, for all $\mathcal{A} \in \mathcal{B}(\mathcal{K}')$,
\begin{align}\label{e:typicalPalm}
\PP(Z \in \mathcal{A}) = \PP^0(Z_0 \in \mathcal{A}),
\end{align}
where $\PP^0$ is the distribution conditioned on a cell center being located at the origin.

The following result shows an important relationship between the distribution of the zero cell and the typical cell of a stationary random mosaic. In particular, that the distribution of $Z_0 - c(Z_0)$ has a Radon-Nikodym derivative with respect to the distribution of $Z$ given by $V(\cdot)/\mathds{E}[V(Z)]$.

\begin{theorem}\label{typical_zero}(Theorem 10.4.1 in \cite{weil})
Let $X$ be a stationary random mosaic in $\R^n$. Denote its typical cell by $Z$ and zero cell by $Z_0$. For any nonnegative measurable and translation invariant function $f: \mathcal{K}' \rightarrow \R$,
\begin{align*}
\mathds{E}[f(Z_0)] = \frac{1}{\mathds{E}[V(Z)]} \mathds{E}[f(Z) V(Z)].
\end{align*}
\end{theorem}

An application of the above result gives the density of a vector uniformly sampled in the centered zero cell of a random mosaic.
\begin{lemma}\label{X_density}
Let $X$ be a stationary random mosaic in $\R^n$ with zero cell $Z_0$ and typical cell $Z$ with respect to a center function $c: \mathcal{C}^{'} \rightarrow \R^n$ as previously defined. Let $Y$ be a random vector in $\R^n$ such that conditioned on $X$,
\begin{align*}
Y \sim \mathrm{Uniform}(Z_0 - c(Z_0)). 
\end{align*}
Then, for all nonnegative measurable functions $g : \R^n \to \R$,
\begin{align*}
\mathds{E}[g(Y)] = \int_{\R^n} g(x) \frac{\mathds{P}(x \in Z)}{\mathds{E}[V(Z)]} \dint x,
\end{align*}
i.e., $Y$ has a density given by $f_Y(x) =  \mathds{P}(x \in Z)/\mathds{E}[V(Z)]$.
\end{lemma}

\begin{proof}
Let $g: \R^n \to \R$ be a nonnegative measurable function. By the definition of $Y$,
\begin{align*}
\mathds{E}[g(Y)] &= \mathds{E}\left[ \mathds{E}[g(Y) | Z_0] \right] =  \mathds{E}\left[ \frac{1}{V(Z_0)} \int_{\R^n} g(x) 1_{\{x \in Z_0 - c(Z_0) \}}\dint x \right].
\end{align*}
Note that the nonnegative function $f: \mathcal{K}' \rightarrow \R$ defined by $f(\cdot) := \frac{1}{V(\cdot)} \int_{\R^n} g(x) 1_{\{x \in \cdot - c(\cdot) \}} \dint x$ is translation invariant, since for any $t \in \R$ and $K \in \mathcal{K}'$,
\begin{align*}
f(K + t) &= \frac{1}{V(K + t)} \int_{\R^n} g(x) 1_{\{x \in K + t - c(K + t)\}}\dint x \\
&= \frac{1}{V(K)} \int_{\R^n} g(x) 1_{\{x \in K + t - c(K) - t \}}\dint x \\
&= \frac{1}{V(K)} \int_{\R^n} g(x) 1_{\{x \in K - c(K)\}}\dint x = f(K).
\end{align*}
Then, by Theorem \ref{typical_zero} and the fact that $c(Z) = 0$,
\begin{align*}
\mathds{E}[g(Y)] &= \mathds{E}\left[ \frac{1}{V(Z_0)} \int_{\R^n} g(x) 1_{\{x \in Z_0 - c(Z_0) \}}\dint x \right] \\
&= \frac{1}{\mathds{E}[V(Z)]} \mathds{E}\left[ V(Z)  \frac{1}{V(Z)} \int_{\R^n} g(x) 1_{\{x \in Z - c(Z) \}}dx \right] \\
&= \mathds{E}\left[ \int_{\R^n} g(x) \frac{1_{\{x \in Z\}}}{\mathds{E}[V(Z)]} \dint x \right]. 
\end{align*}
Applying Fubini's Theorem gives the final result.
\end{proof}

%%%%%%%%%%%%%%%%%%%%%%%%%%%%%%%%%%

\section{Poisson-Voronoi Mosaic}

%%%%%%%%%%%%%%%%%%%%%%%%%%%%%%%%%%

The first type of random mosaic we consider comes from the Voronoi cells of a Poisson point process in $\R^n$. Let $N$ be a stationary Poisson point process with intensity $\lambda$ and, for $x \in N$, define the Voronoi cell of $N$ with center $x$ by
\[C(x, N) := \{ z \in R^n : |z - x| \leq |z - y| \text{ for all } y \in N\}.\]
The collection $X := \{C(x, N) : x \in N\}$ is a stationary random mosaic and is called the Poisson-Voronoi mosaic induced by $N$. The intensity $\lambda$ of the underlying Poisson point process is the cell intensity of the induced mosaic. %, and \[\mathds{E}[V(Z)] = \frac{1}{\lambda}.\]

We first show that for a stationary Poisson-Voronoi mosaic $X$, the density of the random vector $Y$ that is uniformly distributed in $Z_0 - c(Z_0)$, conditioned on $X$, is log-concave. We also compute the moments of its norm. In this case, the center function $c$ assigns to each cell the point of the underlying Poisson point process it corresponds to. 

\begin{prop}\label{p:vor_moment}
Let $Z_0$ be the zero cell of a stationary Poisson-Voronoi mosaic associated to a Poisson point process $N$ with intensity $\lambda$ in $\R^n$. Define the random vector $Y$, such that conditioned on $Z_0$,
%\[ 
$Y \sim \mathrm{Uniform}(Z_0- c(Z_0))$. %\]
Then, $Y$ has the log-concave density
\[f_Y(x) =   \lambda e^{-\lambda\kappa_n |x|^n},\]
and for all $k \in \mathds{N}$,
\[ \mathds{E}[|Y|^k] = \frac{\Gamma(1 + \frac{k}{n})}{(\lambda \kappa_n)^{\frac{k}{n}}}.\] % = O \left( \lambda^{-\frac{k}{n}} n^{\frac{k}{2}}\right). \]
\end{prop}

Note that the moments of $|Y|$ have the following asymptotic approximation: by \eqref{e:vball},
\[ \mathds{E}[|Y|^k] \sim  \frac{n^{\frac{k}{2}}}{\lambda^{\frac{k}{n}} (2\pi e)^{\frac{k}{2}}}, \qquad \text{ as } n \to \infty.\]
The fact that $Y$ has a radial and log-concave density implies $|Y|$ concentrates to 
 \begin{equation}\label{e:sigma}
 \mathds{E}[|Y|^2]^{\frac{1}{2}} = \frac{\Gamma(1 + \frac{2}{n})^{\frac{1}{2}}}{(\lambda \kappa_n)^{\frac{1}{n}}} \sim \frac{\sqrt{n}}{\lambda^{\frac{1}{n}}\sqrt{2\pi e}}, \text{ as } n \to \infty,
 \end{equation}
by the thin-shell estimate \eqref{e:thinshell}. We can also prove strong concentration inequalities by direct computation.

\begin{theorem}\label{t:vor_con}
Define the random vector $Y$ in $\R^n$ as in Proposition \ref{p:vor_moment}.
%Let $X$ be a stationary Poisson-Voronoi mosaic in $\R^n$ with cell intensity $\lambda$ and let $Y$ be a random vector such that, conditioned on $X$, $Y \sim \mathrm{Uniform}(Z_0 - c(Z_0))$. 
Let $\sigma^2 = \mathds{E}|Y|^2$. Then there exists an absolute constant $c > 0$ such that for all $t > 0$,
\[ \mathds{P}\left(|Y| \geq (1 + t)\sigma \right) \leq e^{-ce^{n \ln (1 + t)}},\]
and for all $t \in (0,1)$ and $n \geq 2$,
\[\mathds{P}\left(|Y| \leq (1 - t)\sigma \right) \leq e^{n \ln (1 - t)}.\]
\end{theorem}

Considering a sequence of these vectors in increasing dimensions, we obtain the following threshold result when the cell intensity grows exponentially with dimension.
\begin{theorem}\label{t:vor_thresh}
For each $n$, let $Z_{0,n}$ be the zero cell of a stationary Poisson-Voronoi mosaic in $\R^n$ with cell intensity $e^{n \rho_n}$, and assume $\lim_{n \to \infty} \rho_n = \rho \in \mathds{R}$. Define the random vectors $Y_n$ such that, conditioned on $Z_{0,n}$, $Y_n \sim \mathrm{Uniform}(Z_{0,n} - c(Z_{0,n}))$. Then,
\begin{align*}
\lim_{n \rightarrow \infty} \mathds{P}(|Y_n| \leq \sqrt{n}R) = \begin{cases}
0, & R < e^{-\rho}\left(2\pi e \right)^{-\frac{1}{2}} \\ 1, & R > e^{-\rho}\left(2\pi e \right)^{-\frac{1}{2}}. \end{cases}
\end{align*}
For $R < e^{-\rho}\left(2\pi e \right)^{-\frac{1}{2}}$,
\begin{align*}
\lim_{n \rightarrow \infty} \frac{1}{n} \ln \mathds{P}(|Y_n| \leq \sqrt{n}R) = \rho + \frac{1}{2} \ln (2 \pi e) + \ln R,
\end{align*}
and for $R > e^{-\rho}\left(2\pi e \right)^{-\frac{1}{2}}$,
\begin{align*}
\lim_{n \rightarrow \infty} \frac{1}{n} \ln \left( - \ln \mathds{P}(|Y_n| \geq \sqrt{n}R) \right) = \rho + \frac{1}{2} \ln (2 \pi e) + \ln R.
\end{align*}
\end{theorem}

%\subsection{Proof of Theorem \ref{t:vor_thresh}}

%\begin{proof}

We now provide proofs of the above results. 
\subsection{Proof of Proposition \ref{p:vor_moment}}
By Lemma \ref{X_density}, \eqref{e:EVZ}, and \eqref{e:typicalPalm}, the density of $Y$ satisfies
\begin{align*}
f_Y(x) = \frac{\mathds{P}(x \in Z)}{\E[\text{V}(Z)]} = \lambda \mathds{P}^{0}(x \in Z_0) = \lambda \mathds{P}(N(B_n(x, |x|)) = 0) = \lambda e^{-\lambda\kappa_n |x|^n},
\end{align*}
where $B_n(x, |x|)$ denotes the ball in $\R^n$ with center $x$ and radius $|x|$, and the third equality follows from Slivnyak's theorem \cite[Theorem 3.3.5]{weil}.
This density is clearly log-concave. %Thus, the density of $Y$ is log-concave. 

For the moments, switching to polar coordinates and using another change of variables ($y = \lambda \kappa_n r^n$) gives
\begin{align*}
\mathds{E}[|Y|^k] &= \lambda \int_{\R^n}|x|^k e^{-\lambda \kappa_n |x|^n} \dint x = \lambda n\kappa_n \int_0^{\infty} r^{n + k - 1} e^{- \lambda \kappa_n r^n} \dint r \\
&= \lambda n\kappa_n \int_0^{\infty} \left(\frac{y}{\lambda \kappa_n}\right)^{1 + \frac{k}{n} - \frac{1}{n}} e^{- y} \frac{1}{n\lambda \kappa_n} \left(\frac{y}{\lambda \kappa_n}\right)^{\frac{1}{n} - 1} \dint y \\
&=(\lambda \kappa_n)^{-\frac{k}{n}} \int_0^{\infty} y^{\frac{k}{n}} e^{-y} \dint y =(\lambda \kappa_n)^{-\frac{k}{n}} \Gamma\left(1 + \frac{k}{n}\right).
\end{align*}
%Then, by \eqref{e:vball}, as $n \rightarrow \infty$, 
%\[ \mathds{E}[|Y|^k] \sim  \frac{n^{\frac{k}{2}}}{\lambda^{\frac{k}{n}} (2\pi e)^{\frac{k}{2}}}.\]

%\end{proof}

%\begin{proof}
\subsection{Proof of Theorem \ref{t:vor_con}}
By Proposition \ref{p:vor_moment},
\begin{align}\label{e:Ycdf}
 \mathds{P}(|Y| \leq R) &= \lambda \int_{B_n(R)} e^{-\lambda\kappa_n |x|^n} \dint x 
 = \lambda n \kappa_n \int_0^{R}r^{n-1} e^{-\lambda\kappa_n r^n} \dint r \\  
&= \int_0^{\lambda\kappa_nR^n} e^{- y} \dint y 
= 1 - e^{- \lambda \kappa_n R^n}. \nonumber 
\end{align}
By \eqref{e:sigma}, $\sigma^n = (\lambda \kappa_n)^{-1}\Gamma\left(1 + \frac{2}{n}\right)^{\frac{n}{2}}$, and thus
\[\mathds{P}(|Y| \leq (1 - t)\sigma) = 1 - e^{- \Gamma(1 + \frac{2}{n})^{\frac{n}{2}}(1-t)^n}. \]
Note that $\Gamma(x) \leq 1$ for $x \in [1,2]$. By the inequality $1 - e^{-x} \leq x$ and the assumption $n \geq 2$,
\begin{align*}
\mathds{P}(|Y| \leq (1 - t)\sigma) \leq \Gamma\left(1 + \frac{2}{n}\right)^{\frac{n}{2}}(1-t)^n \leq \Gamma(2)^{\frac{n}{2}} e^{n \ln (1-t)} = e^{n\ln (1-t)}.
\end{align*}
Similarly, by \eqref{e:sigma} and \eqref{e:Ycdf},
\[
\mathds{P}(|Y| \geq (1 + t)\sigma) = 1 - \mathds{P}(|Y| \leq (1 + t)\sigma) %\int_{\lambda \kappa_n R^n}^{\infty} e^{- y} \dint y 
= e^{-  \Gamma(1 + \frac{2}{n})^{\frac{n}{2}} (1 + t)^n}.
\]
By the identity $\Gamma(x + 1) = x\Gamma(x)$, the fact that $\Gamma(x)$ is increasing for $x \geq 2$, and inequality $1 + x \leq e^x$,
\[\Gamma\left(1 + \frac{2}{n}\right)^{\frac{n}{2}} = \left(\frac{\Gamma(2 + \frac{2}{n})}{1 + \frac{2}{n}}\right)^{\frac{n}{2}} \geq \frac{\Gamma(2)^{\frac{n}{2}}}{e} = e^{-1}.\]
Hence, for $c = e^{-1}$,
\[ \PP(|Y| \geq (1 + t)\sigma)  \leq e^{-ce^{n \ln (1 + t)}}.\]
%\end{proof}

%\begin{proof}
\subsection{Proof of Theorem \ref{t:vor_thresh}}
First, by \eqref{e:Ycdf}, 
\begin{align*}
\mathds{P}(|Y_n| \leq \sqrt{n}R) 
= 1 - e^{- e^{n \rho_n}\kappa_n(\sqrt{n}R)^n}.
\end{align*}
For all $R > 0$, by $\eqref{e:vball}$, 
\begin{align*}
e^{n \rho_n}\kappa_n(\sqrt{n}R)^n \sim \frac{1}{\sqrt{\pi n}} \left(e^{\rho_n} (2\pi e)^{\frac{1}{2}} R\right)^n, 
\text{ as } n \to \infty.
\end{align*}
Now assume $R < e^{-\rho}\left(2\pi e \right)^{-\frac{1}{2}}$. The above asymptotic implies $\mathds{P}(|Y_n| \leq \sqrt{n}R) \to 0$ as $n \to \infty$ and there is an $\alpha > 0$ such that for all $n$ large enough,
\begin{align*}
\alpha(e^{n \rho_n}\kappa_n(\sqrt{n}R)^n) \leq 1 - e^{- e^{n \rho_n}\kappa_n(\sqrt{n}R)^n} \leq e^{n \rho_n}\kappa_n(\sqrt{n}R)^n.
\end{align*}
Thus,
\begin{align*}
\lim_{n \rightarrow \infty} \frac{1}{n} \ln \mathds{P}(|Y_n| \leq \sqrt{n}R) = \rho + \frac{1}{2} \ln (2 \pi e) + \ln R.
\end{align*}
Next assume $R > e^{-\rho}\left(2\pi e \right)^{-\frac{1}{2}}$. Then, 
$\mathds{P}(|Y_n| \geq \sqrt{n}R) = e^{- e^{n \rho_n}\kappa_n(\sqrt{n}R)^n} \rightarrow 0$
as $n \rightarrow \infty$, and 
\begin{align*}
\lim_{n \rightarrow \infty} \frac{1}{n} \ln \left( - \ln \mathds{P}(|Y_n| \geq \sqrt{n}R) \right) = \rho + \frac{1}{2} \ln (2 \pi e) + \ln R.
\end{align*}

%\end{proof}

%%%%%%%%%%%%%%%%%%%%%%%%%%%%%%%%%%%%%%%%%%%

\section{Poisson Hyperplane Mosaic}

%%%%%%%%%%%%%%%%%%%%%%%%%%%%%%%%%%%%%%%%%%%
The second type of mosaic considered here is the stationary random mosaic induced by a stationary and isotropic Poisson hyperplane process. % $X$ in $\R^n$. 
A hyperplane process in $\R^n$ is a point process in the space of $(n-1)$-dimensional affine subspaces in $\R^n$, denoted by $\mathcal{H}^n$. 

The following theorem provides a decomposition for the intensity measure of a stationary and isotropic hyperplane process. Note that elements of the space $\mathcal{H}^n$ are of the form $H(u,\tau) := \{x \in \R^n : \langle x, u \rangle = \tau\}$, where $u \in \R^n$ and $\tau \in \R$.

\begin{theorem}\label{t:hyp_int}(Theorem 4.4.2 and (4.33) in \cite{weil})
Let $\hat{X}$ be a stationary and isotropic hyperplane process in $\R^n$ with intensity measure $\Theta \neq 0$. Then, there is a unique number $\gamma \in (0, \infty)$ such that for all nonnegative measurable functions $f$ on $\mathcal{H}^n$,
\[\int_{\mathcal{H}^n} f \dint\Theta =  \gamma \int_{\sS^{n-1}} \int_{-\infty}^{\infty} f(H(u, \tau))\dint \tau \sigma_{n-1}(\dint u).\]
\end{theorem}

The parameter $\gamma$ is called the intensity of the hyperplane process. Its relation to the cell intensity $\lambda$ of the induced random mosaic is given by (see \cite[(10.4.6)]{weil})
\begin{equation}\label{e: cell_int}
 \lambda = \kappa_n\left( \frac{ \gamma\kappa_{n-1}}{n \kappa_n }\right)^n.\end{equation}
%Since the hyperplane process is assumed to be stationary, the induced random mosaic is stationary. 
%The typical cell of this random mosaic has an inradius, that is, the radius of the largest ball completely contained in the cell, given by the following theorem due to R.E. Miles.

%\begin{theorem} (Theorem 10.4.8 in \cite{weil}) Let $\hat{X}$ be a nondegenerate stationary Poisson hyperplane process in $\R^n$ with intensity $\gamma$. Let $Z$ be the typical cell. Then,
%\begin{align*}
%\mathds{P}(r(Z) \leq a) = 1 - e^{-2\gamma a}, \quad a \geq 0.
%\end{align*}
%\end{theorem}

Now, define the center function $c$ to return the center of the inball of a convex polytope, where the inball is the largest ball inside the polytope. A representation for the distribution of the typical cell of a stationary and isotropic Poisson hyperplane mosaic with this center function was obtained by Miles \cite{Miles} in dimension two, and was extended by Calka \cite{Calka_typ1Published} to all dimensions. The result can be described as follows. 

Let $\hat{X}$ be a stationary and isotropic Poisson hyperplane process with intensity $\gamma$ and $X$ be the induced random mosaic. Let $R \in \R^{+}$ and $(U_0, \ldots , U_n) \in (\mathds{S}^{n-1})^{n + 1}$ be independent random variables, where $R$ is exponentially distributed with parameter $2 \gamma$ and $(U_0, \ldots , U_n) $ has a density with respect to the uniform measure that is proportional to the volume of the simplex constructed with these $n+1$ vectors multiplied by the indicator that this simplex contains the origin. In particular, $R$ has the distribution of the inradius of the typical cell of $X$.
Then, let $\hat{X}_R$ be the hyperplane process $\hat{X}$ restricted to $\R^n \backslash B_n(r)$. Let $\mathcal{C}_1$ be the polytope containing the origin obtained by intersecting the $n+1$ halfspaces bounded by the hyperplanes $H(U_i , R)$, $i = 0, \ldots, n$, and $\mathcal{C}_2$ be the zero cell of the mosaic induced by $\hat{X}_R$. Then, the typical cell of $X$ is distributed as $\mathcal{C}_1 \cap \mathcal{C}_2$.

To state the formal result, we first define the halfspace
\[H^{-}(u,r) :=\{ x \in \R^n: \langle x, u \rangle \leq r\}, \qquad  u \in \mathds{S}^{n-1}, \, r \in \R.\]
Also, for a hyperplane $H$, let $H_0^+$ denote the closed halfspace bounded by $H$ that contains the origin. 

\begin{theorem}\label{calka_typical} (Theorem 10.4.6 in \cite{weil}) 
Let $\hat{X}$ be a stationary and isotropic Poisson hyperplane process in $\R^n$ with intensity $\gamma$. Let $\mathds{Q}$ be the distribution of the typical cell $Z$ of the induced random mosaic $X$ with respect to the inball center as the center function. Then, for all Borel sets $A \in \mathcal{B}(\mathcal{K})$,
\begin{align*}
\mathds{Q}(A) &= \mathds{E}[V(Z)]\frac{\gamma^{n+1}}{(n+1)} \int_{0}^{\infty}  \int_{(\sS^{n-1})^{n+1}}  e^{-2\gamma r} \mathds{P}\left( \bigcap_{H \in \hat{X} \cap \mathcal{F}^{B_n(r)}} H_0^+ \cap \bigcap_{j=0}^n H^{-}(u_j, r) \in A\right) \\
& \qquad \cdot \triangle_n(u_0, ..., u_n)1_P(u_0, ..., u_n) \sigma_{n-1}(\dint u_0)...\sigma_{n-1}(\dint u_n) \dint r.
\end{align*}
\end{theorem}

The triangle notation $\triangle_n(u_0, ..., u_n)$ is the $n$-dimensional volume of the convex hull of the vectors $u_0, \ldots, u_n$. 
%\begin{align*}
%\triangle_n(u_0, ..., u_n) &= \frac{1}{n!} \triangledown_n(u_1 - u_0, u_2 - u_0, ..., u_n - u_n),
%\end{align*}
%where $\triangledown_n( v_1, v_2 , ..., v_n )$ is the volume of parallelpiped spanned by the vectors $v_1, \ldots, v_n$. 
The set $P \subset (\sS^{n-1})^{n+1}$ is the set of all $(n+1)$-tuples of unit vectors such that the origin is contained in their convex hull.

For a stationary and isotropic Poisson hyperplane mosaic, we first have the following proposition. For $C \in \mathcal{K}'$, let $c(C)$ be the center of the inball of $C$.
\begin{prop}\label{p:hyp_moments}
Let $Z_0$ be the zero cell of a stationary and isotropic Poisson hyperplane mosaic in $\R^n$ with cell intensity $\lambda$. Let $Y$ be the random vector such that, conditioned on $Z_0$,
%\[ 
$Y \sim \mathrm{Uniform}(Z_0 - c(Z_0))$. %\]
Then, for all $k \geq 0$,
\[ \mathds{E}[|Y|^k] \leq \frac{\Gamma(n + k + 1)}{\Gamma(n+1)2^k} \left(\frac{\kappa_n}{\lambda}\right)^{k/n}.\]
\end{prop}
Next, we consider a sequence of these random vectors in increasing dimension $n$, and obtain the following result. 

\begin{theorem}\label{t:hyp_thresh}
For each $n$, let $X_n$ be a stationary and isotropic Poisson hyperplane mosaic with cell intensity $e^{n \rho_n}$. Assume $\lim_{n \to \infty} \rho_n = \rho \in \mathds{R}$. Let $Z_{0,n}$ be the zero cell of $X_n$, and define the random vectors $Y_n$ such that, conditioned on $X_n$,
\[Y_n \sim \mathrm{Uniform}(Z_{0,n} - c(Z_{0,n})).\]
Then, for all $R > e^{-\rho}\sqrt{\pi e/2}$,
\[\lim_{n \to \infty} \mathds{P}(|Y_n| \geq \sqrt{n} R) = 0,\]
and there is a $R_{\ell}$ such that $0 < R_{\ell} < e^{-\rho}\sqrt{\pi e /2}$ and for all $R < R_{\ell}$,
\[\lim_{n \to \infty} \mathds{P}(|Y_n| \leq \sqrt{n} R) =  0.\]
Also, for $R > e^{-\rho}\sqrt{\pi e / 2}$,%\frac{1}{\rho}$,
\begin{align*}
\limsup_{n \rightarrow \infty} \frac{1}{n} \ln \mathds{P}(|Y_n| \geq \sqrt{n} R) \leq  \rho + \frac{1}{2}\ln\left(\frac{2 e}{\pi}\right) + \ln R  -   e^{\rho} R \sqrt{\frac{2}{\pi e}},%1 - \rho R + \ln (\rho R).
\end{align*}
and for $ R < e^{-\rho}\sqrt{\pi e / 2}$,%\frac{1}{\rho}$,
\begin{align*}
\limsup_{n \rightarrow \infty} \frac{1}{n} \ln \mathds{P}(|Y_n| \leq \sqrt{n} R) \leq  \rho + \frac{1}{2}\ln\left(\frac{ \pi e}{2}\right) + \ln R -   e^{\rho} R\sqrt{\frac{2}{\pi e}}.%1 - \rho R + \ln (\rho R) -  \ln 2 + \ln \pi,
\end{align*}
\end{theorem}

\begin{rem}
The lower bound $R_{\ell}$ satisfies
\[\rho + \frac{1}{2}\ln\left(2 \pi e\right) + \ln R_{\ell} -   e^{\rho} R_{\ell}\sqrt{\frac{2}{\pi e}} - \ln 2 = 0.\]
For all $R > 0$,
%\begin{align*}
$\rho + \frac{1}{2}\ln\left(2 \pi e\right) + \ln R -   e^{\rho} R \sqrt{\frac{2}{\pi e}} - \ln 2 < \rho + \frac{1}{2}\ln\left(2 \pi e\right) + \ln R$,
%\end{align*}
and thus $ \rho + \frac{1}{2}\ln\left(2 \pi e\right) + \ln R_{\ell} > 0$. This implies that 
\[ R_{\ell} > e^{-\rho}(2\pi e)^{-\frac{1}{2}}.\]
Comparing this with Theorem \ref{t:vor_thresh} shows that the vector $Y_n$ chosen with respect to the Poisson-Voronoi zero cell will have a smaller norm in high dimensions than the vector $Y_n$ chosen with respect to the zero cell of the Poisson hyperplane mosaic.
\end{rem}

\begin{rem}
The above result can be generalized to different scalings for the cell intensity and the norm in the following way. Let the cell intensity be equal to $e^{n \rho_n} n^{n\alpha}$ for each $n$ %e^{\lambda_ \rho\frac{n^{\alpha}\kappa_n}{2\kappa_{n-1}}$ 
for some $\alpha \in \R$ and such that $\lim_{n \to \infty} \rho_n = \rho \in \R$. By \eqref{e: cell_int} and \eqref{e:vball}, this implies that the scaling for the intensity of hyperplanes satisfies
\[\gamma_n = \frac{n \kappa_n}{\kappa_{n-1}}\left(\frac{e^{n\rho_n}n^{n\alpha}}{\kappa_n}\right)^{1/n} \sim \frac{e^{\rho}}{\sqrt{e}}n^{\alpha + 1}, \, \text{ as } n \to \infty.\]
The proof and the conclusions of Theorem \ref{t:hyp_thresh} then follow with the %distance scaling like $n^{\alpha - 1}$ instead of $\sqrt{n}$. Indeed, what is needed for the result to hold is that for $\gamma_n R_n \frac{2\kappa_{n-1}}{n\kappa_n} = \rho R n$, and then the results holds for the 
probabilities 
\[\mathds{P}(|Y_n| \leq R n^{\frac{1}{2} - \alpha}) \,  \text{ and } \, \mathds{P}(|Y_n| \geq R n^{\frac{1}{2} - \alpha}).\]
%where $R_n  = Rn^{\frac{1}{2} - \alpha}$.
%This requirement gives two special cases: a cell intensity of $e^{n \rho}$ corresponds to $R_n = O(\sqrt{n})$ and a cell intensity of  $e^{n\rho}n^{\frac{n}{2}}$ corresponds to $R_n = O(1)$.
\end{rem}

Before proving the theorem, recall the following special functions. The beta function is defined by
\begin{equation*}
B(x,y) := \int_0^1 t^{x-1}(1-t)^{y-1} \dint t.
\end{equation*}
The incomplete beta function is defined as $B(x; a,b) := \int_0^x t^{a-1}(1-t)^{b-1} \dint t$, and the regularized incomplete beta function is
\begin{align*}
I_x(a,b) := \frac{B(x; a,b)}{B(a,b)}.
\end{align*}

Recall that the Gamma function is defined by $\Gamma(x) := \int_0^{\infty} t^{x - 1} e^{-t}\dint t$,
and we define the upper and lower incomplete gamma functions by
\begin{align*}
\Gamma_u(x, R) := \int_R^{\infty} t^{x - 1} e^{-t}\dint t \, \text{ and } \, \Gamma_{\ell}(x, R) := \int_0^{R}t^{x - 1} e^{-t}\dint t,
\end{align*}
respectively. The following series of lemmas are needed before proving the results.

%%%%%%%%%%%%%%%%%%%%%%%%%%%%%%%%%%%%%%%%%%
\begin{lemma}\label{l:mean_hit}
Let $\hat{X}$ be a stationary and isotropic Poisson hyperplane process in $\R^n$ with intensity $\gamma$. Let $[0,x]$ denote the line segment between $0$ and the point $x$. Then, for $r \geq 0$,
\[ \Theta\left(\mathcal{F}_{[0,x]}^{B_n(r)}\right) %:= \mathds{E}\left[X\left(\mathcal{F}_{[0,x]}^{B^0(r)}\right)\right] 
= \gamma |x| \left[ \frac{2\kappa_{n-1}}{n \kappa_n }\left(1-\frac{r^2}{|x|^2}\right)^{\frac{n-1}{2}} - \frac{r}{|x|} I_{1 - \frac{r^2}{|x|^2}}\left(\frac{n-1}{2}, \frac{1}{2}\right) \right]1_{\{|x| \geq r\}}.\]
\end{lemma}

\begin{proof}
Note that if $r > |x|$, a hyperplane cannot hit $[0,x]$ and not hit the open ball $B_n(r)$ at the same time and thus $\Theta\left(\mathcal{F}_{[0,x]}^{B_n(r)}\right)$ is zero. If $r \leq |x|$, then by Theorem \ref{t:hyp_int},
\begin{align*}
\Theta\left(\mathcal{F}_{[0,x]}^{B_n(r)}\right) &=  \gamma \int_{\sS^{n-1}} \int_{-\infty}^{\infty} 1_{\{H(u,t) \cap [0,x] \neq 0\}}1_{\{H(u,t) \cap B_n(r) = \emptyset\}} \dint t \sigma_{n-1}(\dint u) \\
&= 2\gamma \int_{\sS^{n-1}} \int_{0}^{\infty} 1_{\{r \leq t < \langle x, u \rangle\}}1_{\{r \leq \langle x,u \rangle\}} \dint t \sigma_{n-1}(\dint u) \\
&= 2\gamma \int_{\{v \in \sS^{n-1}: \langle v, x \rangle \geq r\}} ( \langle x, u \rangle - r) \sigma_{n-1}(\dint u) \\ 
&= 2\gamma |x|  \int_{\{v \in \sS^{n-1}: \langle v, x \rangle \geq r\}} \bigg< \frac{x}{|x|}, u \bigg>  \sigma_{n-1}(\dint u) - 2r\gamma \sigma_{n-1}(\{v \in \sS^{n-1}: \langle v, x \rangle \geq r\}).
\end{align*}
%where $a_+ = \max\{a , 0\}$. 
To compute the first integral, first note that the integral does not depend on the direction of $x$, only on the norm $|x|$. We can then assume $x = |x|e_n$, where $e_n = (0, ..., 0,1)$, and
\begin{align*}
&\int_{\{v \in \sS^{n-1}: \langle v, |x|e_n \rangle \geq r\}} \langle e_n, u \rangle  \sigma_{n-1}(\dint u) =\int_{\{v \in \sS^{n-1}: v_n \geq \frac{r}{|x|}\}} u_n \sigma_{n-1}(\dint u) \\
&\qquad = \frac{\omega_{n-2}}{\omega_{n-1}}\int_{\frac{r}{|x|}}^1 \int_{\sS^{n-2}} t  (1 - t^2)^{\frac{n-3}{2}} \sigma_{n-2}(\dint u) \dint t\\
&\qquad = \frac{(n-1)\kappa_{n-1}}{n \kappa_n }\int_{\frac{r}{|x|}}^1 t  (1 - t^2)^{\frac{n-3}{2}}  \dint t = \frac{(n-1)\kappa_{n-1}}{n \kappa_n } \int_{0}^{1-\frac{r^2}{|x|^2}} \frac{1}{2}s^{\frac{n-3}{2}} \dint s \\
&\qquad =\frac{(n-1)\kappa_{n-1}}{2n \kappa_n } \frac{2}{n-1}\left(1-\frac{r^2}{|x|^2}\right)^{\frac{n-1}{2}} = \frac{\kappa_{n-1}}{n \kappa_n } \left(1-\frac{r^2}{|x|^2}\right)^{\frac{n-1}{2}},
%&= \frac{1}{2} \frac{\Gamma(\frac{n}{2})}{\frac{n-1}{2}\Gamma(\frac{1}{2})\Gamma(\frac{n-1}{2})}\left(1-\frac{r^2}{x^2_n}\right)^{\frac{n-1}{2}} \\
%&= \frac{1}{2 (\frac{n-1}{2})\beta\left(\frac{1}{2}, \frac{n-1}{2}\right)} \left(1-\frac{r^2}{x^2_n}\right)^{\frac{n-1}{2}} 
\end{align*}
where the second equality follows from \cite[(1.41)]{Muller}.
The fractional area of a spherical cap is given by 
\begin{align*}
\sigma_{n-1}(\{v \in \sS^{n-1}: \langle v, x \rangle \geq r\}) = \frac{1}{2}I_{1 - \frac{r^2}{|x|^2}}\left(\frac{n-1}{2}, \frac{1}{2}\right).
\end{align*}
Then,
\begin{align*}
\Theta\left(\mathcal{F}_{[0,x]}^{B(r)}\right) &=\left[ 2 \gamma |x|\frac{\kappa_{n-1}}{n \kappa_n }\left(1-\frac{r^2}{|x|^2}\right)^{\frac{n-1}{2}} - r\gamma I_{1 - \frac{r^2}{|x|^2}}\left(\frac{n-1}{2}, \frac{1}{2}\right) \right]1_{\{|x| \geq r\}} \\
&= \gamma |x| \left[ \frac{2\kappa_{n-1}}{n \kappa_n }\left(1-\frac{r^2}{|x|^2}\right)^{\frac{n-1}{2}} - \frac{r}{|x|} I_{1 - \frac{r^2}{|x|^2}}\left(\frac{n-1}{2}, \frac{1}{2}\right) \right]1_{\{|x| \geq r\}}.
\end{align*}

\end{proof}

\begin{lemma}\label{l:bound}
Let $Z_0$ be the zero cell of the random mosaic induced by a stationary and isotropic Poisson hyperplane process $\hat{X}$ with intensity $\gamma$ in $\R^n$. Let $Y$ be the random vector such that, conditioned on $Z_0$, $Y  \sim \mathrm{Uniform}(Z_0 - c(Z_0))$. Then, for $R > 0$,
\begin{align*}
\mathds{P}(|Y| \geq R) &\leq \frac{\Gamma_{u}\left(n + 1, 2\gamma R \frac{\kappa_{n-1}}{n \kappa_n}\right)}{\Gamma(n+1)},
 %&\leq \frac{n \kappa^2_n }{4^n} \left[ \left(\frac{n\kappa_n}{2\kappa_{n-1}}\right) \Gamma\left(n + 1, 2\gamma R \frac{\kappa_{n-1}}{n \kappa_n}\right) + o(1)\right], 
%\frac{n \kappa_n (n!)^{\frac{1}{2}}}{(n+1)^{\frac{1}{2}}2^{n+1}(2n)^{\frac{n}{2}}} \int_{0}^{1} \frac{\Gamma(n + 1, 2\gamma f_n(t) R)}{f_n(t)^{n+1}} dt,
\end{align*}
and 
\begin{align*}
\mathds{P}(|Y| \leq R) & \leq \frac{n \kappa_n^2 }{4^{n}}\left(\frac{n \kappa_n}{2\kappa_{n-1}}\right) \left[\Gamma_{\ell}\left(n + 1, 2\gamma R \frac{\kappa_{n-1}}{n \kappa_n}\right) + \Gamma(n)\left(\frac{\kappa_{n-1}}{n\kappa_n}\right)^{n+1} \right].
%\frac{n \kappa_n (n!)^{\frac{1}{2}}}{(n+1)^{\frac{1}{2}}2^{n+1}(2n)^{\frac{n}{2}}} \int_{0}^{1} \frac{\gamma(n + 1, 2\gamma f_n(t) R)}{f_n(t)^{n+1}} dt.
\end{align*}
\end{lemma}

%%%%%%%%%%%%%%%%%%%%%%%%%%%%%%%%%%%%%%%%%%%

\begin{proof}

By Lemma \ref{X_density}, the density of $Y$ is $f_Y(x) = \frac{\mathds{P}(x \in Z)}{\mathds{E}\left[V(Z) \right]}$.
Using the representation of $Z$ in Theorem \ref{calka_typical}, 
\begin{align*}
\mathds{P}(x \in Z) &= \frac{\mathds{E}[V(Z)] \gamma^{n+1}}{n+1} \int_{0}^{\infty}  \int_{(\sS^{n-1})^{n+1}}  e^{-2\gamma r} \mathds{P}\left( x \in \bigcap_{H \in \hat{X} \cap \mathcal{F}^{B_n(r)}} H_0^+ \cap \bigcap_{j=0}^n H^{-}(u_j, r) \right) \\
& \qquad \cdot \triangle_n(u_0, ..., u_n)1_P(u_0, ..., u_n)\prod_{i=0}^n \sigma_{n-1}(\dint u_i) \dint r.
\end{align*}
First, we see that
\begin{align*}
\mathds{P}\left( x \in \bigcap_{H \in \hat{X} \cap \mathcal{F}^{B_n(r)}} H_0^+ \cap \bigcap_{j=0}^n H^{-}(u_j, r)\right) &=\prod_{j=0}^n 1_{\{x \in H^{-}(u_j, r)\}} \mathds{P}\left(\hat{X} \left(\mathcal{F}_{[0,x]}^{B_n(r)}\right) = 0\right) \\
&=  \prod_{j=0}^n 1_{\{x \in H^{-}(u_j, r)\}}e^{-\Theta\left(\mathcal{F}_{[0,x]}^{B_n(r)}\right)}.
\end{align*}
Then, 
\begin{align*}
&\mathds{P}(|Y| \geq R)  = \int_{B_n(R)^c} \frac{\mathds{P}(x \in Z)}{\mathds{E}[V(Z)]} \dint x \\
&= \int_{B_n(R)^c} \frac{\gamma^{n+1}}{(n+1) }\int_{0}^{\infty}  \int_{(\sS^{n-1})^{n+1}}  e^{-2\gamma r} \mathds{P}\left( x \in \bigcap_{H \in \hat{X} \cap \mathcal{F}^{B_n(r)}} H_0^+ \cap \bigcap_{j=0}^n H^{-}(u_j, r) \right) \\
& \qquad \qquad \cdot \triangle_n(u_0, ..., u_n)1_P(u_0, ..., u_n)\prod_{i=0}^n \sigma_{n-1}(\dint u_i) \dint r \dint x \\
&=\frac{\gamma^{n+1}}{n+1}  \int_{B_n(R)^c} \int_{0}^{\infty}  \int_{(\sS^{n-1})^{n+1}}  e^{-2\gamma r} e^{-\Theta\left(\mathcal{F}_{[0,x]}^{B_n(r)}\right)} \triangle_n1_P\prod_{i=0}^n 1_{\{x \in H^{-}(u_i, r)\}}\sigma_{n-1}(\dint u_i) \dint r \dint x\\
&= \frac{\gamma^{n+1}}{n+1}  \int_{B_n(R)^c}\int_{0}^{\infty} e^{-2\gamma r}  e^{-\Theta\left(\mathcal{F}_{[0,x]}^{B_n(r)}\right)}  \int_{(\sS^{n-1})^{n+1}}  \triangle_n 1_P \prod_{i=0}^n 1_{\{\langle x, u_i \rangle \leq r\}} \sigma_{n-1}(\dint u_i) \dint r \dint x. 
\end{align*}
Making the change of variables $t = \frac{r}{|x|}$, observing that the innermost integral does not depend on the direction of $x$, and using Fubini's Theorem gives
\begin{align*}
&\mathds{P}(|Y| \geq R) \\
&=  \frac{\gamma^{n+1}}{n+1}  \int_{B_n(R)^c} |x| \int_{0}^{\infty}  e^{-2\gamma |x| t}  e^{-\Theta\left(\mathcal{F}_{[0,x]}^{B_n(|x| t )}\right)}  \int_{(\sS^{n-1})^{n+1}} \triangle_n 1_P \prod_{i=0}^n 1_{\{\langle e_n, u_i \rangle \leq t\}} \sigma_{n-1}(\dint u_i) \dint t \dint x \\
&=   \frac{\gamma^{n+1}}{n+1} \int_{0}^{\infty} \int_{B_n(R)^c}  |x|  e^{-2\gamma |x| t}  e^{-\Theta\left(\mathcal{F}_{[0,x]}^{B_n(|x| t )}\right)} \dint x  \int_{(\sS^{n-1})^{n+1}} \triangle_n 1_P \prod_{i=0}^n 1_{\{\langle e_n, u_i \rangle \leq t\}} \sigma_{n-1}(\dint u_i) \dint t.
\end{align*}
By Lemma \ref{l:mean_hit} and a change to polar coordinates,
\begin{align*}
& \int_{B_n(R)^c} |x|  e^{-2\gamma |x| t}  e^{-\Theta\left(\mathcal{F}_{[0,x]}^{B_n(|x| t )}\right)} \dint x \\
&\qquad = \int_{B_n(R)^c} |x|  e^{ -2\gamma |x| t - \gamma |x| \left[ \frac{2\kappa_{n-1}}{n \kappa_n }\left(1-t^2\right)^{\frac{n-1}{2}} - t I_{(1 -t^2)}\left(\frac{n-1}{2}, \frac{1}{2}\right) \right]1_{\{t \leq 1\}}  } \dint x  \\
&\qquad =n \kappa_n\int_{R}^{\infty} r^{n}  e^{-\gamma r \left[2t + \left(\frac{2\kappa_{n-1}}{n \kappa_n} \left(1-t^2\right)^{\frac{n-1}{2}} -  t I_{1 -  t^2 }\left(\tfrac{n-1}{2}, \tfrac{1}{2} \right) \right)1_{\{t \leq 1\}}\right] } \dint r.
\end{align*}
Now, using the identity $I_{1-x}(a,b) = 1 - I_x(b,a)$, 
\begin{align*}
\int_{B_n(R)^c} |x|  e^{-2\gamma |x| t}  e^{-\Theta\left(\mathcal{F}_{[0,x]}^{B_n(|x| t )}\right)} \dint x &=  n \kappa_n \int_{R}^{\infty} r^n e^{-\gamma r f_n(t) } \dint x,
\end{align*}
where
\begin{align*} 
f_n(t) &:= \begin{cases}  t + \frac{2\kappa_{n-1}}{n \kappa_n} \left(1-t^2\right)^{\frac{n-1}{2}} + t I_{t^2 }\left(\tfrac{1}{2},\tfrac{n-1}{2} \right) , & 0 \leq t \leq 1 \\ 2t, & t \geq 1. \end{cases}
%&= \frac{\kappa_{n-1}}{n \kappa_n}  \left( \frac{n\kappa_n}{\kappa_{n-1}}t + \left(1-t^2\right)^{\frac{n-1}{2}} -  \frac{n \kappa_n}{2\kappa_{n-1}} tI_{1 -  t^2 }\left(\tfrac{n-1}{2}, \tfrac{1}{2} \right) \right)
\end{align*}
Note that $f_n$ attains its minimum at $t = 0$ and that
\begin{equation}\label{e:f_zero}
f_n(0)=  \frac{2\kappa_{n-1}}{n\kappa_n}.
\end{equation}
Then, by the change of variables $y = \gamma f_n(t) r$,
\begin{align*} 
n \kappa_n\int_{R}^{\infty} r^{n} e^{- \gamma f_n(t) r} \dint r 
&= \frac{n \kappa_n}{ \left(\gamma f_n(t) \right)^{n+1}}\int_{\gamma f_n(t) R}^{\infty} y^{n} e^{-y} \dint y  = \frac{n \kappa_n}{ \left(\gamma f_n(t) \right)^{n+1}} \Gamma_u(n + 1, \gamma f_n(t) R).
\end{align*} 
This gives us that
\begin{align}\label{e:Ycdf_hyp}
&\mathds{P}(|Y| \geq R) = \frac{n \kappa_n }{n+1} \int_{0}^{\infty} \frac{\Gamma_u(n + 1, \gamma f_n(t) R)}{ f_n(t)^{n+1}} \int_{(\sS^{n-1})^{n+1}}  \triangle_n 1_P \prod_{i=0}^n 1_{\{\langle e_n, u_i \rangle \leq t\}}\sigma_{n-1}(\dint u_i) \dint t. 
\end{align}

Since the upper incomplete gamma function is decreasing in its second argument, for all $t \geq 0$, 
\[\Gamma_u(n+1, \gamma f_n(t) R) \leq \Gamma_u\left(n+1, \gamma  \frac{2\kappa_{n-1}}{n\kappa_n} R\right),\]
where we have used  \eqref{e:f_zero}. This gives the upper bound
\begin{align*}
&\mathds{P}(|Y| \geq R) \leq \frac{\Gamma_u(n + 1, \gamma  \frac{2\kappa_{n-1}}{n\kappa_n} R)}{\Gamma(n+1)} \\
& \qquad \cdot \left[  \frac{ n \kappa_n }{n+1} \int_{0}^{\infty} \frac{\Gamma(n+1)}{f_n(t)^{n+1}}\left( \int_{(\sS^{n-1})^{n+1}}  \triangle_n 1_P \prod_{i=0}^n 1_{\{\langle e_n, u_i \rangle \leq t\}} \sigma_{n-1}(\dint u_i)\right) \dint t \right].
\end{align*}
The term in the parentheses is the right hand side of \eqref{e:Ycdf_hyp} for $R = 0$, %of the integral $\int_{\R^n} \frac{\mathds{P}(x \in Z)}{\mathds{E}[V(Z)]} dx$ 
and is thus equal to 1. Hence, 
\begin{align*}
\mathds{P}(|Y| \geq R) &\leq  \frac{\Gamma_u\left(n + 1, \gamma R \frac{2\kappa_{n-1}}{n \kappa_n}\right)}{\Gamma(n+1)}.
\end{align*}

For the upper bound of $\mathds{P}(|Y| \leq R)$, we follow a similar procedure up to the equality
\begin{align*}
\mathds{P}(|Y| \leq R) &= \frac{n \kappa_n }{n+1} \int_{0}^{\infty} \frac{\Gamma_{\ell}(n + 1, \gamma f_n(t) R)}{ f_n(t)^{n+1}} \int_{(\sS^{n-1})^{n+1}}  \triangle_n 1_P \prod_{i=0}^n 1_{\{\langle e_n, u_i \rangle \leq t\}} \sigma_{n-1}(\dint u_i) \dint t .
\end{align*}
The lower incomplete gamma function is not decreasing in $t$ like the upper incomplete gamma function, so one cannot proceed exactly as above. Instead we first use the upper bound 
%\[ \mathds{E}\left[\prod_{j=0}^n 1_{\{\langle e, U_j \rangle \leq t\}} \triangle_n(U_0, \ldots, U_n) 1_P(U_0, \ldots, U_n)\right] \leq  \mathds{E}\left[\triangle_n(U_0, \ldots, U_n) 1_P(U_0, \ldots, U_n)\right].\]
\[\int_{(\sS^{n-1})^{n+1}} \triangle_n 1_P \prod_{i=0}^n 1_{\{\langle e_n, u_i \rangle \leq t\}} \sigma_{n-1}(\dint u_i) \leq \int_{(\sS^{n-1})^{n+1}} \triangle_n 1_P \prod_{i=0}^n \sigma_{n-1}(\dint u_i).\]
Then, by the fact that
\[\frac{n2^n \omega_n^{n-1}}{(n+1)\kappa_{n-1}^n} \triangle(u_0, \ldots, u_n)1_P(u_0, \ldots, u_n) \prod_{i=0}^n \sigma_{n-1}(\dint u_i)\]
is a joint density (see equation (11) in \cite{Calka_typ1}), we have
\[\int_{(\sS^{n-1})^{n+1}} \triangle_n(u_0, \ldots, u_n) 1_P(u_0, \ldots, u_n) \prod_{i=0}^n \sigma_{n-1}(\dint u_i) =  \frac{\kappa_n(n+1)}{2^n} \left(\frac{\kappa_{n-1}}{n \kappa_n}\right)^n,\]
where we have used the fact that $\omega_n = n \kappa_n$. This gives 
\begin{align*}
\mathds{P}(|Y| \leq R) \leq \frac{n \kappa^2_n }{2^n}\left(\frac{\kappa_{n-1}}{n \kappa_n}\right)^n\int_{0}^{\infty} \frac{\Gamma_{\ell}(n + 1, \gamma f_n(t) R)}{ f_n(t)^{n+1}}  \dint t .
\end{align*}

Now, note that for $t \geq 1$, $f_n(t) = 2t$, so % and $1_{\{\langle e_n, u_j \rangle \leq t\}} = 1$ for each $j$
\begin{align*}
\int_1^{\infty} \frac{\Gamma_{\ell}(n + 1, \gamma  f_n(t) R)}{f_n(t)^{n+1}}  \dint t \leq \frac{1}{2^{n+1}}\int_1^{\infty} \frac{\Gamma(n+1)}{t^{n+1}} \dint t  = \frac{\Gamma(n)}{2^{n+1}}.% \sim \frac{(2\pi n)^{-\frac{n+1}{2}}}{n}.
\end{align*}
Thus, 
\begin{align*}
\mathds{P}(|Y| \leq R) \leq\frac{n \kappa^2_n }{2^n}\left(\frac{\kappa_{n-1}}{n \kappa_n}\right)^n \left[ \int_{0}^{1} \frac{\Gamma_{\ell}(n + 1, \gamma f_n(t) R)}{ f_n(t)^{n+1}}  \dint t + \frac{\Gamma(n)}{2^{n+1}} \right] .
\end{align*}
Next we show that the function
\[h_n(t) :=  \frac{\Gamma_{\ell}(n + 1, \gamma f_n(t) R)}{f_n(t)^{n+1}}\]
is decreasing and thus reaches its maximum at $t = 0$.  
It suffices to show $h_n'(t) \leq 0$. Indeed, we first observe that the derivative of $f_n'(t)$,
\begin{align*}
f_n'(t) = \begin{cases} 1 + I_{t^2}\left(\frac{1}{2}, \frac{n-1}{2}\right), & 0 \leq t \leq 1 \\ 2, & t \geq 1, \end{cases}
\end{align*}
is positive. 
Then, by the Fundamental Theorem of Calculus,
\begin{align*}
\frac{\dint}{\dint t}\Gamma_{\ell}(n + 1,  \gamma f_n(t) R) = e^{- \gamma f_n(t) R}\left(\gamma f_n(t) R\right)^n \left(\gamma R \right)f_n'(t),
\end{align*}
and by the quotient rule,
\begin{align*}
h_n'(t) &= \frac{1}{f_n(t)^{2n + 2}} \left( f_n(t)^{n+1}\frac{\dint}{\dint t}\Gamma_{\ell}(n + 1,  \gamma f_n(t) R)  - (n+1)f_n(t)^n f_n'(t)\Gamma_{\ell}(n + 1,  \gamma f_n(t) R)  \right) \\
%&= -  \frac{1}{f_n(t)^{n+2}} \left(  (n+1) f_n'(t)\Gamma_{\ell}(n + 1,  \gamma f_n(t) R) -  e^{- \gamma f_n(t) R}\left(\gamma f_n(t) R\right)^{n+1} f_n'(t) \right) \\
&=  -  \frac{f_n'(t)}{f_n(t)^{n+2}} \left( (n+1) \Gamma_{\ell}(n + 1,  \gamma  f_n(t)R)- e^{- \gamma f_n(t) R}\left(\gamma f_n(t) R\right)^{n+1}  \right). %\\\leq 0.
\end{align*}
Since $f_n$ and $f_n'$ are positive, it suffices to show the following inequality for $h'$ to be negative:
\begin{align}\label{e:ineq1}
 e^{- \gamma f_n(t) R}\left(\gamma f_n(t) R\right)^{n+1} \leq (n +1)\Gamma_{\ell}(n+1,  \gamma f_n(t) R)  .
\end{align}
Indeed, since $e^{-t} \geq e^{-x}$ for all $t \in [0,x]$,
\begin{align*}
\Gamma_{\ell}(n+1, x) &= \int_0^x e^{-t}t^n \dint t  \geq e^{-x} \int_0^x t^{n} \dint t = e^{-x}\frac{x^{n+1}}{n+1}.
\end{align*}
Letting $x = 2\gamma f(t) R$ gives \eqref{e:ineq1}, and hence $h'(t) \leq 0$.
Thus, by \eqref{e:f_zero},
\begin{align*}
\mathds{P}(|Y| \leq R) \leq   \frac{n \kappa_n^2 }{4^{n}}\left(\frac{n \kappa_n}{2\kappa_{n-1}}\right) \left[\Gamma_{\ell}\left(n + 1, 2\gamma R \frac{\kappa_{n-1}}{n \kappa_n}\right) + \Gamma(n)\left(\frac{\kappa_{n-1}}{n\kappa_n}\right)^{n+1} \right].
\end{align*}

\end{proof}

%%%%%%%%%%%%%%%%%%%%%%%%%%%%%%%%%%%
We now prove the main results. 
\subsection{Proof of Proposition \ref{p:hyp_moments}}

By Lemma \ref{l:bound},
\begin{align*}
\mathds{E}[|Y|^k] &= k \int_0^{\infty} y^{k-1} \mathds{P}(|Y| \geq y) \dint y  \\
&\leq k \int_0^{\infty} y^{k-1} \frac{\Gamma_u(n+1, 2\gamma \frac{\kappa_{n-1}}{n \kappa_n} y)}{\Gamma(n+1)} \dint y \\
& = \frac{k}{\Gamma(n+1)} \int_0^{\infty} y^{k-1} \left(\int_{2\gamma \frac{\kappa_{n-1}}{n \kappa_n} y}^{\infty} t^n e^{-t}\dint t \right)\dint y.
\end{align*}
Then, by Fubini's Theorem,
\begin{align*}
\mathds{E}[|Y|^k] &\leq \frac{k}{\Gamma(n+1)} \int_0^{\infty} t^n e^{-t} \left(\int_0^{\frac{t}{2\gamma}\frac{n\kappa_n}{\kappa_{n-1}}} y^{k-1} \dint y \right)\dint t \\
&= \frac{1}{\Gamma(n+1)} \int_0^{\infty} t^n e^{-t} \left(\frac{n \kappa_n}{2\gamma \kappa_{n-1}} t\right)^k \dint t = \frac{\Gamma(n + k + 1)}{\Gamma(n+1)}  \left(\frac{n \kappa_n}{2\gamma \kappa_{n-1}} \right)^k.
\end{align*}
Applying \eqref{e: cell_int} gives the final conclusion:
\[ \mathds{E}[|Y|^k] \leq \frac{\Gamma(n + k + 1)}{\Gamma(n+1)2^k} \left(\frac{\kappa_n}{\lambda}\right)^{k/n}.\]

%%%%%%%%%%%%%%%%%%%%%%%%%%%%%%%%%%%

\subsection{Proof of Theorem \ref{t:hyp_thresh}}

%%%%%%%%%%%%%%%%%%%%%%%%%%%%%%%%%%%

%\begin{proof}
To prove the results, we will compute asymptotic approximations on the bounds from Lemma \ref{l:bound}. These bounds depend on the intensity of the hyperplane process associated to the random hyperplane mosaic, so we first observe the relationship of this intensity to that of the cell intensity in this setting. 

Let $\gamma_n$ denote the intensity of the Poisson hyperplane process $\hat{X}_n$ that corresponds to the Poisson hyperplane mosaic $X_n$ with cell intensity $e^{n \rho_n}$. By \eqref{e: cell_int} and \eqref{e:vball}, $\gamma_n$ satisfies
\begin{align}\label{e:gamma_asymptotic}
\gamma_n = \frac{n \kappa_n}{\kappa_{n-1}} \left(\frac{e^{n \rho_n}}{\kappa_n}\right)^{\frac{1}{n}} \sim \frac{e^{\rho}}{\sqrt{e}} n \, \text{ as } \, n \to \infty.
\end{align}
Define the constants $c_n := \gamma_n \sqrt{n}R \frac{2\kappa_{n-1}}{n^2 \kappa_n}$ and the limit
\[c := \lim_{n \to \infty} c_n = \lim_{n \to \infty} \gamma_n \sqrt{n}R \frac{2\kappa_{n-1}}{n^2 \kappa_n} =\lim_{n \to \infty} \frac{e^{\rho}}{\sqrt{e}} n \sqrt{n} R \frac{2}{n\sqrt{ 2 \pi n}} = e^{\rho} R\sqrt{\frac{2}{\pi e}},\]
where the second equality follows from \eqref{e:vball} and \eqref{e:gamma_asymptotic}. Then, by a modified application of Laplace's method (see \ref{A:gamma}), for $c > 1$,
\begin{equation}\label{e:gamu_asymp}
\Gamma_u(n +1 , 2\gamma_n \sqrt{n}R \frac{\kappa_{n-1}}{n\kappa_n}  ) = 
\Gamma_u(n + 1 , c_n n)  \sim n^n \frac{c e^{n(\ln c_n - c_n)}}{(c - 1)},
\end{equation}
and for $c  < 1$,
\begin{equation}\label{e:gaml_asymp}
\Gamma_{\ell}(n + 1, 2\gamma_n \sqrt{n}R \frac{\kappa_{n-1}}{n\kappa_n}  ) = \Gamma_{\ell}(n + 1 , c_n n)  \sim n^n \frac{c e^{n(\ln c_n -  c_n)}}{(1- c)}. 
\end{equation}

First assume $R > e^{-\rho} \sqrt{ \pi e /2}$, which is equivalent to $c > 1$. By Lemma \ref{l:bound}, \eqref{e:gamu_asymp}, and Stirling's formula, as $n \rightarrow \infty$,
\begin{align*}
\mathds{P}(|Y_n| \geq \sqrt{n}R) &\leq \frac{\Gamma_u\left(n + 1, c_n n \right)}{\Gamma(n+1)} 
 \sim \frac{1}{\sqrt{2\pi n}}\left(\frac{e}{n}\right)^n n^n \frac{c e^{n(\ln c_n - c_n)}}{(c - 1)} =  \frac{c e^{n(\ln c_n - c_n + 1) }}{\sqrt{2\pi n}(c - 1)},
%&\leq \frac{n \kappa_n^2 }{2(4^{n})}\left(\frac{n \kappa_n}{\kappa_{n-1}}\right)  \Gamma\left(n + 1, \rho Rn \right) \\ 
%& \sim \frac{n}{4^{n}(2\pi n)}\left(\frac{2\pi e}{n}\right)^n (2\pi n)^{\frac{1}{2}}n^n \frac{(\rho R)^{n+1} e^{- n\rho R}}{n(\rho R - 1)} \\
%&= \left(\frac{\pi}{2}\right)^n \frac{(\rho R)^{n+1} e^{- n(\rho R - 1) }}{\sqrt{2\pi n}(\rho R - 1)}.
\end{align*}
%Then, for $R > e^{-\lambda} \frac{\sqrt{\pi e}}{\sqrt{2}}$,
and thus
\begin{align*}
\limsup_{n \rightarrow \infty} \frac{1}{n} \ln \mathds{P}(|Y_n| \geq \sqrt{n}R) &\leq \limsup_{n \to \infty} (\ln c_n - c_n + 1) = \ln c - c + 1\\
%&= -  \frac{e^{\rho} R\sqrt{2}}{\sqrt{\pi e}} + \ln  \frac{e^{\rho} R\sqrt{2e}}{\sqrt{\pi}} \\
&= \rho + \frac{1}{2}\ln\left(\frac{2 e}{\pi }\right) + \ln R  -   e^{\rho} R\sqrt{\frac{2}{ \pi e}}  .%-  \ln 2 + \ln \pi.
\end{align*}
Now assume that $R  < e^{-\rho} \sqrt{\pi e /2}$, which is equivalent to $c < 1$. By Lemma \ref{l:bound},
\begin{align*}
\mathds{P}(|Y_n| \leq \sqrt{n}R)&\leq\frac{n \kappa^2_n }{4^n}\left(\frac{n\kappa_n}{2\kappa_{n-1}}\right) \Gamma_{\ell}\left(n + 1, c_n n \right)\left( 1 +  \frac{\Gamma(n)\left(\frac{\kappa_{n-1}}{n\kappa_n}\right)^{n+1} }{\Gamma_{\ell}\left(n + 1, c_n n \right)}\right).
\end{align*}
Note that by \eqref{e:gaml_asymp} and Stirling's formula, 
\begin{align*}
\frac{\Gamma_{\ell}\left(n + 1, c_n n \right)}{\Gamma(n+1)} 
 \sim  \frac{c e^{n(\ln c_n - c_n + 1) }}{\sqrt{2\pi n}(c - 1)}, \,  \text{ as } n \rightarrow \infty.
\end{align*}
Thus by \eqref{e:vball},
\begin{align*}
\lim_{n \rightarrow \infty} \frac{\Gamma(n)\left(\frac{\kappa_{n-1}}{n\kappa_n}\right)^{n+1} }{\Gamma_{\ell}\left(n + 1, c_n n \right)} &= \lim_{n \rightarrow \infty} \frac{1}{n}\left(\frac{\kappa_{n-1}}{n\kappa_n}\right)^{n+1} \frac{\Gamma(n+1)}{\Gamma_{\ell}\left(n + 1, c_n n \right)} = 0.
\end{align*}
Also, again using  \eqref{e:vball} and \eqref{e:gaml_asymp}, 
\begin{align*}
\frac{n \kappa^2_n }{4^n}\left(\frac{n\kappa_n}{2\kappa_{n-1}}\right) \Gamma_{\ell}\left(n + 1, c_n n \right) \sim  \left(\frac{\pi}{2}\right)^n\left(\frac{n}{2\pi}\right)^{\frac{1}{2}} \frac{c e^{n(\ln c_n - c_n + 1) }}{(1-c)}, \,  \text{ as } n \rightarrow \infty.
\end{align*}
Putting these observations together gives
\begin{align*}
\limsup_{n \rightarrow \infty} \frac{1}{n} \ln \mathds{P}(|Y_n| \leq \sqrt{n}R) \leq \rho + \frac{1}{2}\ln\left(\frac{\pi e}{2}\right) + \ln R  -   e^{\rho} R\sqrt{\frac{2}{\pi e}}.
\end{align*}

%\end{proof}

%%%%%%%%%%%%%%%%%%%%%%%%%%%%%%%%%%%%%%
%%%%%%%%%%%%%%%%%%%%%%%%%%%%%%%%%%%%%%
\appendix
%%%%%%%%%%%%%%%%%%%%%%%%%%%%%%%%%%%%%%

%%%%%%%%%%%%%%%%%%%%%%%%%%%%

\section{Laplace Method}

%%%%%%%%%%%%%%%%%%%%%%%%%%%%

\begin{lemma}\label{A:laplace}
Let $a, b \in \R$ such that $a <b$ and let $\{a_n\}_{n \in \mathds{N}}$ be a sequence of real numbers such that $\lim_{n \to \infty} a_n = a$. Consider a function $f(t)$ that attains its minimum at $t = a$ on the interval $[a , b)$ and $f'(t)$ is continuous. If $f'(a) > 0$, then as $n \rightarrow \infty$,
\begin{align*}
\int_{a_n}^b e^{-n f(t)} \dint t \sim  \frac{e^{-n f(a_n)}}{n f'(a)}.
\end{align*}
Now suppose $f(t)$ attains its minimum at $t = b$ over the interval $(a,b]$ and $\{b_n\}_{n \in \mathds{N}}$ is a sequence of real numbers such that $\lim_{n \to \infty} b_n = b$. If $f'(b) < 0$, then as $n \to \infty$,
\begin{align*}
\int_a^{b_n} e^{-n f(t)} \dint t \sim  - \frac{e^{-n f(b_n)}}{n f'(b)}.
\end{align*}
\end{lemma}

\begin{proof}
To prove the first statement, let $\varepsilon > 0$ such that $f'(a) - \varepsilon > 0$. By the continuity of $f'$, there exists $\delta > 0$ such that $|t - a| < 2\delta$ implies $f'(t) \leq f'(a) + \varepsilon$. In addition, $\delta$ can be made small enough so that $\delta < \frac{b-a}{2}$. Then, there exists $M$ such that for all $n \geq M$, $\delta < b-a_n$ and $|a - a_n| < \delta$. Fix $n \geq M$. By Taylor's theorem, for each $t \in [a_n,b]$, there is some $\xi_t \in (a_n, t)$ such that 
\[ f(t) = f(a_n) + f'(\xi_t)(t - a_n).\]
Then, for all $t$ such that $|t - a_n| < \delta$, we have that $|\xi_t - a_n| < \delta$, and the triangle inequality gives $|\xi_t - a|  <  2\delta$. This implies 
\[ f(t) \leq f(a_n) + (f'(a) + \varepsilon)(t - a_n).\]
Since the integrand is positive, 
\begin{align*}
\int_{a_n}^b e^{-n f(t)} \dint t &\geq \int_{a_n}^{a_n + \delta} e^{-n f(t)} \dint t \geq \int_{a_n}^{a_n + \delta} e^{-n(f(a_n) + (f'(a) + \varepsilon)(t - a_n))} \dint t  \\
&= \frac{e^{-n f(a_n)}}{n(f'(a) + \varepsilon)} \int_0^{\delta n (f'(a) + \varepsilon)} e^{- y} \dint y = \frac{e^{-n f(a_n)}}{n(f'(a) + \varepsilon)} \left(1 - e^{- \delta n (f'(a) + \varepsilon)}\right).
\end{align*}
Then, since this inequality holds for all $n \geq M$,
\begin{align*}
\liminf_{n \rightarrow \infty} \frac{\int_{a_n}^b e^{-n f(t)} \dint t }{\frac{e^{-n f(a_n)}}{n(f'(a) + \varepsilon)} } \geq \liminf_{n \rightarrow \infty} \left(1 - e^{- \delta n (f'(a) + \varepsilon)}\right) = 1.
\end{align*}

%For $\delta > 0$, let $N$ be such that for all $n > N$, 
For the upper bound, a similar Taylor series argument gives that there exists $\delta > 0$ such that for all $n$ large enough and $t$ such that $|t - a_n| < \delta$,
\[f(t) \geq f(a_n) + (f'(a) - \varepsilon)(t - a_n).\]
Then, define
\[ C := \inf_{t \in [a + \frac{\delta}{2}, b]} f(t) > f(a),\]
where the lower bound follows from the assumption that $f$ achieves its minimum at $a$ on the interval $[a,b)$, and define $\eta := C - f(a) > 0$. Note that for all $t \in [a + \frac{\delta}{2}, b]$, $f(t) > f(a) + \eta$. Then, for all $n$ large enough,
\begin{align*}
\int_{a_n}^b e^{-n f(t)} \dint t &= \int_{a_n}^{a_n + \delta} e^{-n f(t)} \dint t  +  \int_{a_n + \delta}^b e^{-n f(t)} \dint t  \\
&\leq \int_{a_n}^{a_n + \delta} e^{-n (f(a_n) + (f'(a) - \varepsilon)(t - a_n)} \dint t  +  \int_{a + \frac{\delta}{2}}^b e^{-n C} \dint t   \\
& < (b - a) e^{-n C}  + \frac{e^{-nf(a_n)}}{n(f'(a) - \varepsilon)} \int_0^{\delta n (f'(a) - \varepsilon)} e^{-y} \dint y \\
&=  (b - a) e^{-n C}  + \frac{e^{-nf(a_n)}}{n(f'(a) - \varepsilon)} \left( 1 - e^{- \delta n (f'(a) - \varepsilon)} \right),
\end{align*}
where the first upper bound uses that fact that for all $n$ large enough, $b > a_n + \delta > a + \frac{\delta}{2}$. This upper bound implies
\begin{align*}
\limsup_{n \rightarrow \infty} \frac{\int_{a_n}^b e^{-n f(t)} \dint t }{\frac{e^{-n f(a)}}{n(f'(a) - \varepsilon)} } \leq \limsup_{n \rightarrow \infty} \{(b-a)n(f'(a) - \varepsilon)e^{-n\eta} + 1 - e^{- \delta n (f'(a) - \varepsilon)} \} = 1,
\end{align*}
since $\eta > 0$ and $f'(a) - \varepsilon > 0$. These limits hold for all $\varepsilon$ small enough, and thus,
\begin{align*}
\lim_{n \rightarrow \infty} \frac{\int_{a_n}^b e^{-n f(t)} \dint t }{\frac{e^{-n f(a_n)}}{nf'(a)} } = 1.
\end{align*}
The second statement of the lemma is proved similarly.

\end{proof}

\begin{lemma}\label{A:gamma} 
Let $\{c_n\}_{n \in \mathds{N}}$ be a sequence of positive real numbers and assume that $\lim_{n \to \infty} c_n = c \in (0, \infty)$. Then, as $n \rightarrow \infty$, if $c > 1$,
\[ \Gamma_u(n+1, c_n n) \sim n^n \frac{ce^{n( \ln c_n- c_n)}}{(c-1)},\]
and if $c < 1$,
\[ \Gamma_{\ell}(n+1, c_n n) \sim n^n \frac{ce^{n(\ln c_n - c_n )}}{(1-c)}.\]
\end{lemma}

\begin{proof}
By a change of variables,
\[ \Gamma_u(n+1, c_n n) = \int_{c_nn}^{\infty} e^{-t}t^{n} \dint t =n^{n+1} \int_{c_n}^{\infty} e^{-n(y - \ln y)} \dint y.\]
For $c > 1$, the function $f(y) = y - \ln(y)$ has a minimum at $c$ on the interval $[c, \infty)$. Also observe that $f'(y) = 1 - \frac{1}{y}$, and for $c > 1$, $f'(c) = 1 - \frac{1}{c} > 0$. Thus, by Lemma \ref{A:laplace},
\begin{align*}
\Gamma_u(n+1, c_n n)  \sim n^{n+1} \frac{e^{-n(c_n - \ln c_n)}}{n\left(1 - \frac{1}{c}\right)} = n^{n} \frac{c e^{n(\ln c_n - c_n)}}{(c- 1)}.% = n^n \frac{c^{n+1}e^{-nc}}{n(c-1)}.
\end{align*}
Similarly, 
\[\Gamma_{\ell}(n+1, c_n n) = \int_0^{c_nn} e^{-t}t^n \dint t =n^{n+1} \int_0^{c_n} e^{-n(y - \ln y)} \dint y,\]
and for $c < 1$, $f(y) = y - \ln(y)$ hits its minimum at $c$ on the interval $(0, c]$ and $f'(c) = 1 - \frac{1}{c} < 0$. Thus, again by Lemma \ref{A:laplace},
\[ \Gamma_{\ell}(n+1, c_n n) \sim n^n \frac{c e^{n(\ln c_n - c_n)}}{(1- c)}.\]%= n^n \frac{c^{n+1}e^{-nc}}{n(1-c)}.\]
\end{proof}

%\nocite{*}
\bibliographystyle{abbrv}
\bibliography{Ref_zerocell}

\end{document}